\documentclass[twoside,12pt]{article}
\usepackage{geometry}
\usepackage{amsmath,amssymb,amsfonts, amsthm}
\usepackage{txfonts}
\usepackage{todonotes}
\usepackage{graphicx}
\usepackage{mathrsfs}
\usepackage{color}
\usepackage{float}
\usepackage{extarrows}
\usepackage{exscale}
\usepackage{titlesec}
\titleformat{\section}[block]{\Large\center\sc}{\arabic{section}}{0.5em}{}[]
\usepackage{relsize}
\usepackage{color}
\usepackage{hyperref}

\usepackage[titles]{tocloft}
\usepackage{fancyhdr}
\theoremstyle{plain}
\newtheorem{theorem}{Theorem}[section]
\newtheorem{lemma}[theorem]{Lemma}

\newtheorem{proposition}[theorem]{Proposition}

\theoremstyle{definition}
\newtheorem{definition}[theorem]{Definition}

\theoremstyle{remark}
\newtheorem{remark}[theorem]{Remark}

\let\oldsection\section
\renewcommand\section{\setcounter{equation}{0}\oldsection}

\def\be{\begin{equation}}
\def\ee{\end{equation}}
\def\bes{\begin{equation*}}
\def\ees{\end{equation*}}
\def\bs{\begin{split}}
\def\es{\end{split}}
\def\bali{\begin{aligned}}
\def\eali{\end{aligned}}
\newcommand{\pf}{\noindent {\bf Proof. \hspace{2mm}}}
\def\bR{{\mathbb R}}

\def\al{\alpha}
\def\e{\epsilon}

\def\la{\lambda}

\def\t{\tilde}
\def\th{\theta}

\def\g{\gamma}

\def\dl{\delta}

\def\lt{\left}
\def\rt{\right}
\def\ls{\lesssim}

\def\i{\infty}
\def\p{\partial}
\def\f{\frac}

\def\s{\sqrt}
\def\q{\quad}
\def\qq{\qquad}

\allowdisplaybreaks

\def\angt{\langle t\rangle}
\def\angr{\langle r \rangle}

\def\bN{\mathbb N}

\begin{document}

\title{\normalsize\bf GLOBAL TANGENTIALLY ANALYTICAL SOLUTIONS OF THE 3D AXIALLY SYMMETRIC PRANDTL EQUATIONS}

\author{\normalsize\sc Xinghong Pan and Chao-Jiang Xu}

\date{}

\maketitle

\begin{abstract}

In this paper, we will prove the global existence of solutions to the three dimensional axially symmetric Prandtl boundary layer equations with small initial data, which lies in $H^1$ Sobolev space with respect to the normal variable and is analytical with respect to the tangential variables. Proof of the main result relies on the construction of a tangentially weighted analytic energy functional, which acts on a specially designed good unknown. The constructed energy functional can find its two dimensional parallel in Ignatova-Vicol \cite{IV:2016ARMA} where no tangential weight is introduced and the specially good unknown is set to control the lower bound of the analytical radius, whose two dimensional similarity can be traced to Paicu-Zhang \cite{PZ:2021ARMA}. Our result is an improvement of that in Ignatova-Vicol \cite{IV:2016ARMA} from the almost global existence to the global existence and an extension of that in Paicu-Zhang \cite{PZ:2021ARMA} from the two dimensional case to the three dimensional axially symmetric case.

\medskip



\end{abstract}


\section{Introduction}

\q\ The main purpose of this paper is to study the well-posedness of the initial-boundary value problem for the three dimensional axially symmetric Prandtl boundary layer equations in the domain $\{(t, x, y, z)\in \mathbb{R}^4; \ t>0, (x, y) \in \bR^2, z>0\}$.

The general three dimensional Prandtl boundary layer equations read as follows,

\be\label{3dprandtl}
\lt\{
\bali
&\partial_{t} \t{u}+\left(\t{u} \partial_{x}+\t{v} \partial_{y}+\t{w} \partial_{z}\right) \t{u}+\partial_{x} p=\partial_{z}^{2} \t{u}, \\
&\partial_{t} \t{v}+\left(\t{u} \partial_{x}+\t{v} \partial_{y}+\t{w} \partial_{z}\right) \t{v}+\partial_{y} p=\partial_{z}^{2} \t{v}, \\
&\partial_{x} \t{u}+\partial_{y} \t{v}+\partial_{z} \t{ w}=0,\\
&(\t{u}, \t{v}, \t{w})\big|_{z=0}=0, \quad \lim_{z \rightarrow+\infty}(\t{u}, \t{v})=(U(t, x, y), V(t, x, y)),
\eali
\rt.
\ee
where $(U(t, x, y), V(t, x, y))$ and $p(t, x, y)$ are the tangential velocity fields and pressure of the Euler flow, satisfying
\be\label{3deuler}
\left\{
\bali
&\partial_{t} U+U \partial_{x} U+V \partial_{y} U+\partial_{x} p=0, \\
&\partial_{t} V+U \partial_{x} V+V \partial_{y} V+\partial_{y} p=0.
\eali
\right.
\ee
Here we write $\t{\boldsymbol{u}}=(\t{u},\t{v},\t{w})$ and $\boldsymbol{U}=(U(t, x, y), V(t, x, y))$.

The Prandtl equations was proposed by Prandtl \cite{Prandtl:1904MATHCONGRESS} in 1904 in order to explain the mismatch between the no slip boundary condition of the Navier-Stokes equations and the corresponding Euler equations when the vanishing viscosity limit $\nu\rightarrow 0$. Reader can see \cite{OS:1999AMMC} and references therein for more introductions on the boundary layer theory and check \cite{GN:2011CPAM} for some recent development on this topic.

Since the Prandtl equations \eqref{3dprandtl} has no tangential diffusion and the advection term will cause one order tangential derivative loss when we perform finite-order energy estimates.  Local in time well-posedness of the Prandtl equations in Sobolev spaces for general data without structure assumptions is still an open question.

For data in Sobolev spaces, under the monotonic assumption on the tangential velocity of the outflow, Oleinik and Samokhin \cite{OS:1999AMMC} proved the local existence and uniqueness by using Crocco transform for the two dimensional Prandtl equations. Recently, in \cite{AWXY:2015JAMS} (see also  \cite{MW:2015CPAM}), the second author of the present work and their collaborators introduce a nice change of variable such that the cancellation property of the bad term was discovered and the local well-posedness in Sobolev spaces was proved by direct weighted energy estimates.  Ill-posedness in Sobolev spaces for the Prandtl equations around non-monotonic outflow can be found in E and Engquist \cite{EE:1997CPAM}, Gerard-Varet and Dormy \cite{GVD:2010JAMS}, and Gerard-Varet and Nguyen \cite{GVN:2012ASYMA}. For the three dimensional Prandtl equations, Liu, Wang and Yang \cite{LWY:2017ADVANCES} proved the local wellposedness of solutions in Sobolev spaces under some constraints on the flow structure in addition to the monotonic assumption. While this flow structure is violated,  in \cite{LWY:2016ARMA}, they showed  the ill posedness of the 3D Prandtl equations in Sobolev spaces, which indicates that the monotonicity condition on tangential velocity fields is not sufficient for the well-posedness of the three-dimensional Prandtl
equations.

  As for the long time behavior of the Prandtl equations in Sobolev spaces, Oleinik and Samokhin \cite{OS:1999AMMC} shows global regular solutions exist when the tangential variable belongs to a finite interval with the amplitude being small. Xin and Zhang \cite{XZ:2004ADVANCES} proved the global existence of weak solutions under an additional favorable sign condition on the pressure $p$. The second author of the present paper and Zhang \cite{XZ:2017JDE} proved that the lifespan of the solution is $\mathcal{O}(\ln \f{1}{\e})$ if the initial data is a small $\e$ perturbation around the monotonic shear flow in Sobolev spaces.  All the above results are discussed in the two dimensional spaces.

For data in analytical spaces, Sammartino and Caflisch \cite{SC:1998CMP} established the local well-posedness in both tangential and normal variables by using the abstract Cauchy-Kowalewski theorem. The analyticity on the normal variable was removed in \cite{LCS:2003SIAM}. Later in \cite{KV:2013CMS}, Kukavica and Vicol gave an energy-based proof of the local well-posedness result with data analytical only with respect to the tangential variable. The above results are both valid for the two and three dimensional Prandtl equations. To relax the analyticity condition is not easy. In the case where the data has
 a single non-degenerate critical point in the normal variable at each fixed tangential variable point, G\'{e}rard-Varet and Masmoudi \cite{GVM:2015ASENS}
 proved the local well-posedness of the two dimensional Prandtl equations  in Gevrey class 7/4 with respect to the tangential variable, which was extended to Gevrey class 2 in \cite{LY:2020JEMS} for data that are small perturbations of a shear flow with a single non-degenerate critical point for the three dimensional Prandtl equations. Note that this exponent $2$ is optimal in view of the instability mechanism of \cite{GVD:2010JAMS}.  Recently, Dietert and G\'{e}rard-varet \cite{DGV:2019ANNPDE} improved the well-posedness to Gevrey class 2 by removing the hypothesis on the number and order of the critical points for the two dimensional Prandtl equations, which was extended to the three dimensional case in Li, Masmoudi and Yang in \cite{LMY:2021ARXIV}.

For the long time existence of the Prandtl equations with analytical data, the first result appeared in Zhang and Zhang \cite{ZZ:2016JFA} where authors there proved that the lifespan of the tangentially analytical solution is $\mathcal{O}(\e^{-4/3})$ if the data is an $\e$ size and the outflow is of size $\e^{5/3}$ for the two and three dimensional Prandtl equations. Later, an almost global existence result was proved in \cite{IV:2016ARMA} in two dimensional case, where a good unknown combining the tangential component of the velocity and its derivative on the normal variable is introduced to extend the existence time. This result was extended to the three case in \cite{LZ:2020AAM}. Most recently, Global existence of tangentially analytical solutions with small data was proved in \cite{PZ:2021ARMA} for the two dimensional Prandtl equations. This result was improved to the optimal Gevrey class 2 in \cite{WWZ:2021ARXIV}. As far as the authors know, there isn't any results concerning on the global existence of tangentially analytical solutions for the three Prandtl equations.

The main purpose of this paper is to study the global existence of tangentially analytical solutions for the three dimensional axially symmetric Prandtl equations. As far as the authors know, study on the axially symmetric flow has attracted more and more attention recently, such as pointwise blow-up criteria and Liouville type theorems for the axially symmetric Navier-Stokes equations in \cite{CSTY:2009CPDE, KNSS:2009ACTAMATH, Pan:2016JDE, CPZZ:2020ARMA} and references therein.  Most recently, Albritton, Bru\'{e} and Colombo obtained the non-uniqueness of Leray solutions of the forced axially symmetric Navier-Stokes equations in \cite{ABC:2021ARXIV}. The novelty of our present work lies in the followings: First, inspiring by the tangentially analytical energy functional in Ignatova-Vicol \cite{IV:2016ARMA}, we will construct a similar energy functional with the main difference being that  the analytical energy constructed in our results involves in a polynomial weight on the tangential variables, which results from the special structure of the axially symmetric Prandtl equations and mainly set to overcome the order mismatch between the tangentially radial velocity, $u^r$, and the normal velocity, $u^z$, with respect to the distance to the symmetric axis, $r$, when we use the divergence free condition to connect them each other.  Second, the unknown acted on by the energy functional is specially designed, which is a combination of the tangentially radial velocity, $u^r$, and its primitive one in the normal variable. This quantity has a sufficiently fast decay-in-time rate for our constructed weighted analytical energy, which ensures the positive lower bound of the analytical radius for any time. Its two dimensional originality can be traced to Paicu-Zhang \cite{PZ:2021ARMA}.


\section{Reformation of the problem and the main theorem}

\subsection{Reformation of the equations}
\q\ In the following, we give a derivation of the three dimensional axially symmetric Prandlt equations in cylindrical coordinates $(r,\,\th,\,z)$, i.e., for $\boldsymbol{x}=(x,\,y,\,z)\in\mathbb{R}^3$,
\bes
r=\sqrt{x^2+y^2},\q \th=\arctan\frac{y}{x},
\ees
a solution of \eqref{3dprandtl} and \eqref{3deuler} are said to be an axisymmetic solution, if and only if
\[
\bali
&\t{\boldsymbol{u}}=\t{u}^r(t,r,z)e_r+\t{u}^{\th}(t,r,z)e_{\th}+\t{u}^z(t,r,z)e_z,\\
&\boldsymbol{U}=U^r(t,r,z)e_r+U^\th (t,r,z)e_{\th},\\
&p=p(t,r),
\eali
\]
satisfy the system \eqref{3dprandtl} and \eqref{3deuler}, separately, where the components of $\t{\boldsymbol{u}}$ and $\t{\boldsymbol{U}}$ in cylindrical coordinates are independent of $\th$ and the basis vectors $e_r,e_\th,e_z$ are
\[
e_r=\left(\frac{x}{r},\frac{y}{r},0\right),\quad e_\th=\left(-\frac{y}{r},\frac{x}{r},0\right),\quad e_z=(0,0,1).
\]

Then in cylindrical coordinates, system \eqref{3dprandtl} and equations \eqref{3deuler} satisfy

\be\label{3dasprandtl}
\lt\{
\bali
&\partial_{t} \t{u}^r+\left(\t{u}^r \partial_{r}+\t{u}^z \partial_{z}\right) \t{u}^r-\f{(\t{u}^\th)^2}{r}+\partial_{r} p=\partial_{z}^{2} \t{u}^r, \\
&\partial_{t} \t{u}^\th+\left(\t{u}^r \partial_{r}+\t{u}^z \partial_{z}\right) \t{u}^\th+\f{\t{u}^\th\t{u}^r}{r}=\partial_{z}^{2} \t{u}^\th, \\
&\f{\partial_{r} (r\t{u}^r)}{r}+\partial_{z} \t{u}^z=0,\\
&(\t{u}^r, \t{u}^\th, \t{u}^z)\big|_{z=0}=0, \quad \lim_{z \rightarrow+\infty}(\t{u}^r, \t{u}^\th)=(U^r, U^\th),
\eali
\rt.
\ee
and

\bes
\left\{
\bali
&\partial_{t} U^r+U^r \partial_{r} U-\f{U^2_\th}{r}+\partial_{r} p=0, \\
&\partial_{t} U^\th+U^r \partial_{r} U^\th+\f{U^rU^\th}{r}=0.
\eali
\right.
\ees

Now we consider that the flow is swirl free, which means $u^\th= U^\th \equiv 0$. Also we consider the simple case of the outflow $U^r\equiv 0$, which indicates that $\p_r p\equiv 0$.   Then \eqref{3dasprandtl} is simplified to

\be\label{3dasprandtl2}
\lt\{
\bali
&\partial_{t} \t{u}^r+\left(\t{u}^r \partial_{r}+\t{u}^z \partial_{z}\right) \t{u}^r-\partial_{z}^{2} \t{u}^r=0, \\
&\f{\partial_{r} (r\t{u}^r)}{r}+\partial_{z} \t{u}^z=0,\\
&\left.(\t{u}^r, \t{u}^z)\right|_{z=0}=0, \quad \lim_{z \rightarrow+\infty} \t{u}^r=0.
\eali
\rt.
\ee

This simplified axially symmetric boundary layer equations \eqref{3dasprandtl2} has appeared in \cite[Chapter 4.1]{OS:1999AMMC}. If the axially symmetric velocity $\t{u}=\t{u}^r(t,r,z)e_r+\t{u}^{\th}(t,r,z)e_{\th}+\t{u}^z(t,r,z)e_z$ is smooth and divergence free, we can deduce that

\bes
\t{u}^r\big |_{r=0}=\t{u}^\th\big|_{r=0}\equiv 0.
\ees
See \cite{LW:2009SIAM}. Then  there isn't singularity for the quantity $\t{u}^r/r$ at $r=0$.

 Set  the new unknowns
\bes
(u^r, u^z):= (\f{\t{u}^r}{r}, \t{u}^z),
\ees
which satisfy the following new formation of axially symmetric Prandtl boundary layer equations
\be\label{3dasprandtl3}
\lt\{
\bali
&\partial_{t} {u}^r+\left(r{u}^r \partial_{r}+{u}^z \partial_{z}\right) {u}^r-\partial_{z}^{2} {u}^r+(u^r)^2=0, \\
&r\p_ru^r+2u^r+\partial_{z} {u}^z=0,\\
&({u}^r, {u}^z)\big|_{z=0}=0, \quad \lim_{z \rightarrow+\infty} {u}^r=0.
\eali
\rt.
\ee
%

\subsection{The linearly good unknown}

We assume $u^r, u^z$ decay sufficiently fast at $z$ infinity and define
\be\label{defnphi}
\phi(t,r,z):=-\int^{+\i}_z u^r(t,r,\bar{z})d \bar{z},
\ee
which also decays sufficiently fast at $z$ infinity. By integrating $\eqref{3dasprandtl3}_1$ on $[z,+\i]$ with respect to $z$ variable, we have
\bes
\lt\{
\bali
&\partial_{t} \phi-\partial_{z}^{2} \phi-u^ru^z+\int^{\i}_z (u^r)^2d\bar{z}-2\int^{\i}_z \p_zu^r u^zd\bar{z}=0, \\
&\p_z\phi\big|_{z=0}=0, \quad \lim_{z \rightarrow+\infty} \phi=0,\\
&\phi\big|_{t=0}=\phi_0=\int^\i_z u^r(0,r,\bar{z})d\bar{z}.
\eali
\rt.
\ees
And $(u^r,u^z)$ is obtained from $\phi$ as
\bes
u^r=\p_z \phi, \q u^z=-r\p_r\phi-2\phi.
\ees
Inspired by the good unknown in \cite{PZ:2021ARMA}, we define
\be\label{defng}
g:=\p_z\phi+\f{ z}{2\langle t\rangle} \phi= u^r+\f{ z}{2\langle t\rangle} \phi,
\ee
which satisfies
\be\label{3dasprandtl5}
\lt\{
\bali
&\partial_{t} g+(ru^r\p_r +u^z\p_z)g-\partial_{z}^{2} g+\f{1}{\angt}g+(u^r)^2-\f{1}{2\angt}u^z\p_z(z\phi)+\f{z}{\angt}u^r\phi\\
 &\qq\qq +\f{z}{2\angt}\int^{\i}_z (u^r)^2d\bar{z}-\f{z}{\angt}\int^{\i}_z \p_zu^r u^zd\bar{z}=0, \\
&g\big|_{z=0}=0, \quad \lim_{z \rightarrow+\infty} g=0,\\
&g\big|_{t=0}=g_0=u^r(0,r,z)+\f{z}{2}\phi_0(r,z).
\eali
\rt.
\ee

The introduced $g$ can control the velocity $u^r$ and $u^z$ nicely with a lower order time weight which leads to the possibility of closing our energy functional defined below for any $t>0$. This good unknown $g$ can be viewed as a lift of that in \cite{IV:2016ARMA}, where the type of good unknown $\t{g}=\p_zu^r+\f{ z}{2\langle t\rangle} u^r$ are introduced to prove the almost global existence of tangentially analytical solutions.

\subsection{Energy functional spaces and the main result}

Set
\bes
\th(t,z):=\exp\lt(\f{ z^2}{8\angt}\rt).
\ees
For $\la \in \bR$, set
\bes
\th_{\la}(t,z)=\exp\lt(\f{ \la z^2}{8\angt}\rt).
\ees
Then for any $\la, \mu\in\bR$, $\th_{\la+\mu}=\th_\la \cdot\th_\mu$.

Denote
\bes
M_n=\f{(n+1)^4}{n!},\ \p^\al_{h}=\p^{\al_1}_{x}\p^{\al_2}_{y},\ \al=(\al_1,\al_2)\in\bN^2,
\ees
and
 \bes
 \angr= (r^2+1)^{1/2}=\s{x^2+y^2+1}, \q \angt=(t+1),\q (x, y)\in\mathbb{R}^2,\q t\ge 0.
 \ees
 For a positive time-dependent function $\tau:=\tau(t)$, we introduce the Sobolev weighted semi-norms
\be\label{element}
\bali
&X_n=X_n(g,\tau)=\sum_{|\al|=n}\|\th \angr^n \p^\al_h g\|_{L^2} \tau^n M_n,\ n\in \bN;\\
&D_n=D_n(g,\tau)=\sum_{|\al|=n}\|\th \angr^n \p^\al_h\p_z g\|_{L^2} \tau^n M_n=X_n(\p_zg,\tau),\ n\in \bN;\\
&Y_n=Y_n(g,\tau)=\sum_{|\al|=n}\|\th \angr^n \p^\al_h g\|_{L^2} \tau^{n-1}n M_n,\ n\in (\bN/\{0\}).
\eali
\ee

We consider the following functional space that is real-analytic in $\boldsymbol{x}_h=(x,y)$ and lies in a weighted $L^{2}$ space with respect to $z$,
$$
\mathcal{X}_{\tau}=\left\{ \forall \al\in\bN^2, \angr^{|\al|}\p^\al_h g(t, r, z) \in L^{2}\left(\bR^3_+ ; \theta^2dxdydz\right):\|g\|_{\mathcal{X}_{\tau}}<\infty\right\}
$$
where
$$
\|g\|_{\mathcal{X}_{\tau}}=\sum_{n \geqq 0} X_{n}(g, \tau).
$$

\begin{remark}
In our definition of the element \eqref{element}$_1$, there is a weight $\angr^{n}$ for the tangential $n$th order derivative, which is set to match and control the term $r\p_r g$ appeared in the equation \eqref{3dasprandtl5}.
\end{remark}

We also define the semi-norm
$$
\|g\|_{\mathcal{Y}_{\tau}}=\sum_{n \geqq 1} Y_{n}(g, \tau),
$$
which encodes the one-derivative gain in the analytic estimates. Note that for $\beta>1$, we have
$$
\|g\|_{\mathcal{Y}_{\tau}} \leqq \tau^{-1}\|g\|_{\mathcal{X}_{\beta \tau}} \sup _{n \geqq 1}\left(n \beta^{-n}\right) \leqq C_{\beta} \tau^{-1}\|g\|_{\mathcal{X}_{\beta \tau}}.
$$
The gain of a $z$ derivative shall be encoded in the dissipative semi-norm
$$
\|g\|_{\mathcal{D}_{\tau}}=\sum_{n \geqq 0} D_{n}(g, \tau)=\left\|\partial_{z} g\right\|_{\mathcal{X}_{\tau}}.
$$

Having introduced the functional spaces in our paper and before presenting the main results, we give a definition of solutions to the reformulated Prandtl equation \eqref{3dasprandtl5}.
\begin{definition}[Classical in tangential variables and weak in normal variable]
 For a fixed time $t>0$, let $\mathcal{H}$ be the closure of the set of functions
\bes
\lt\{f(t,x,y,z)\in C^\i_c(\bR^2\times [0,+\i));\ f|_{z=0}=0 \rt\}
\ees
under the space norm
\bes
\|f(t)\|^2_{\mathcal{H}}:=\sum\limits_{|\al|\leq 3}\int_{\bR^3_+}|\p^\al_h f(t,x,y,z)|^2\exp\lt(\f{z^2}{4\angt} \rt)dxdydz.
\ees

For $T>0$, we say a function $g$ is a classical in $x,y$ and weak in $z$ solution of \eqref{3dasprandtl5} if
\bes
\|g(t)\|_{\mathcal{H}}\in L^\i([0,T))\ \text{ and }\ \|\p_z g(t)\|_{\mathcal{H}}\in L^2([0,T)),
\ees
and \eqref{3dasprandtl5} holds when tested by $C^\i_c([0,T)\times\bR^2\times [0,+\i))$.
\end{definition}

\begin{theorem}\label{thmain}
Let $g_0(r,z)$ be tangentially analytical with radius of analyticity being $\tau_0>0$.  Then, for any $0<\dl\leq \f{1}{4}$, there exists a $\e_0$, depending only on $\dl$ and $\tau_0$, such that for any $\e\leq \e_0$, if
  \bes
  \|g_0\|_{X_{\tau_0}}\leq \e,
  \ees
then \eqref{3dasprandtl5} has a globally in-time solution $g$, which is tangentially  analytical with the radius of analyticity $\tau(t)\geq \f{1}{2}\tau_0$ and for any $t>0$, it satisfies
\be\label{energyest}
\bali
&\langle t\rangle^{\f{5}{4}-\dl}\|g(t)\|_{\mathcal{X}_{\tau(t)}}+\f{\dl}{12}\int^t_0\lt(\langle s\rangle^{\f{1}{4}-\dl}\|g(s)\|_{\mathcal{X}_{\tau(s)}}+\langle s\rangle^{\f{3}{4}-\dl}\|g(s)\|_{\mathcal{D}_{\tau(s)}}\rt)ds\\
&+C_0\int^t_0\f{\langle s\rangle^{\f{5}{4}-\dl}}{\tau^{2}(s)}\lt(\|g(s)\|_{\mathcal{X}_{\tau(s)}}+\langle s\rangle^{1/4}\|g(s)\|_{\mathcal{D}_{\tau(s)}}\rt)\|g(s)\|_{\mathcal{Y}_{\tau(s)}}ds
\leq  \|g_0\|_{\mathcal{X}_{\tau_0}}\leq \e_0.
\eali
\ee
\end{theorem}

\begin{remark}
It follows from the estimates in Lemma \eqref{gcontrol1} and Lemma \eqref{gcontrol2} below that bounds on $g,\ \p_zg$ in \eqref{energyest} in $\mathcal{X}_{\tau}$ imply similar estimates on $u^r$ and $u^z$. So global existence and uniqueness of tangentially analytical solutions in Theorem \ref{thmain} indicates global existence and uniqueness of tangentially analytical solutions for the original system \eqref{3dasprandtl3} and \eqref{3dasprandtl2}.  The proof of Theorem \ref{thmain} mainly consists of a priori estimates (cf. Section \ref{apriori}) and the local well posedness. Since the local existence and uniqueness of the tangentially analytical solutions has already shown in many references, e. g. \cite{KV:2013CMS, IV:2016ARMA, ZZ:2016JFA}, here we only present the a priori estimate \eqref{energyest}.
\end{remark}

\begin{remark}
The construction of the energy functional $\mathcal{X}_\tau$ is inspired by that in \cite{IV:2016ARMA}. The main difference is that there is a weight $\angr^{n}$ for the tangential $n$th order derivative due the appearance of the operator $r\p_r$ in the transport term of the equation \eqref{3dasprandtl5}$_1$.
\end{remark}

\begin{remark}
In the model \eqref{3dasprandtl2}, we only consider the case that the outflow $U^r\equiv 0$. Actually the proof can be also applied to the case that $U^r=r\e f(t)$, where $\e>0$ is sufficiently small and $f(t)$ decays sufficiently fast as $t\rightarrow +\i$. The computation will be more elaborated and complicated. For simplicity and convenience of presenting the main idea, we omit this extension and leave it to the interested reader.
\end{remark}
\qed
\begin{remark}
Here we only consider the the axially symmetric Prandtl equation, 
 extensions of Theorem \ref{thmain} to the axially symmetric MHD boundary layer system and in the tangential Gevrey spaces will be considered in our future work.
\end{remark}

For a function $f(t,x,y,z)$ and $1\leq p,q \leq +\i$,  define
\bes
\|f(t)\|_{L^p_hL^q_z}:=\lt(\int^{+\i}_0 \lt(\int_{\bR^2} |f(t,x,y,z)|^p dxdy\rt)^{q/p}dz \rt)^{1/q}.
\ees

If $p=q$, we simply write it as $\|f\|_{L^p}$ and besides, if $p=q=2$, we will simply denote it as $\|f\|$. Throughout the paper, $C_{a,b,c,...}$ denotes a positive constant depending on $a,\,b,\, c,\,...$ which may be different from line to line. We also apply $A\lesssim_{a,b,c,\cdots} B$ to denote $A\leq C_{a,b,c,...}B$. For a two dimensional multi-index $\al=(\al_1,\al_2)\in\bN^2$, we write $\p^\al_{h}=\p^{\al_1}_{x}\p^{\al_2}_{y}$ and
$\p^k_h=\{\p^\al_h;\ |\al|=k\}$. For a norm $\|\cdot\|$, we use $\|(f,g,\cdots)\|$ to denote $\|f\|+\|g\|+\cdots$.

\section{A priori estimates and proof of the main theorem}\label{apriori}

\q\ First, we state a simple version of the local well-posedness result on the three dimensional Prandtl equations in tangentially analytical spaces. See \cite[Theorem 3.1 and Remark 3.3]{KV:2013CMS}.

\begin{theorem}[Theorem 3.1 of \cite{KV:2013CMS} with the outflow being zero in three dimensional spaces]\label{thlocal}
Fix the constant $\nu>1/2$ and denote $\langle z\rangle:=1+z$. For a function $f(t,x,y,z)$ and $\tau(t)>0$, define
\bes
\|f(t)\|^2_{\t{\mathcal{X}}_{\tau(t)}}:= \sum_{n\geq 0}\sum_{|\al|=n}\|\langle z\rangle^{\nu}\partial^\alpha_{x, y}f(t,x,y,z)\|^2_{L^2(\bR^3_+)}\tau^{2n}(t)M^2_n.
\ees

Then, for $\tau_0>0$, if the solution in \eqref{3dprandtl} with the outflow $\boldsymbol{U}$ being zero satisfies
 \bes
(\t{u},\t{v})|_{t=0}:=(\t{u}_0,\t{v}_0)\in \t{\mathcal{X}}_{\tau_0},
 \ees
then there exists a $T_\ast=T_\ast(\nu, \tau_0, \|(\t{u}_0,\t{v}_0)\|_{\t{\mathcal{X}}_{\tau_0}})>0$, such that the three dimensional Prandtl equations \eqref{3dprandtl} have a unique real-analytical solution in $[0,T_\ast)$ satisfying for any $t\in[0,T_\ast)$, $\tau(t)>0$ and
 \bes
 \|(\t{u},\t{v})(t)\|_{\t{\mathcal{X}}_{\tau(t)}}<+\i.
\ees
\end{theorem}

Based on the above local well-posedness result of the three dimensional Prandtl equations, The proof of Theorem \ref{thmain} is simplified to the following a prior estimate, stated as Proposition \ref{proapriori}, and continuity argument.

\begin{proposition}\label{proapriori}
For $T>0$, let $g$ be the tangentially  analytical solution of \eqref{3dasprandtl5} and  $g_0(r,z)$ be tangentially analytical with radius of analyticity being $\tau_0>0$.  Then, for any $0<\dl\leq \f{1}{4}$, there exists a $\e_0$, depending only on $\dl$ and $\tau_0$ such that for any $\e\leq \e_0$, if
  \bes
  \|g_0\|_{\mathcal{X}_{\tau_0}}\leq \e,
  \ees
then for any $0<t<T$, the solution $g$ satisfies
\bes
\bali
&\langle t\rangle^{\f{5}{4}-\dl}\|g(t)\|_{\mathcal{X}_{\tau(t)}}+\f{\dl}{12}\int^t_0\lt(\langle s\rangle^{\f{1}{4}-\dl}\|g(s)\|_{\mathcal{X}_{\tau(s)}}+\langle s\rangle^{\f{3}{4}-\dl}\|g(s)\|_{\mathcal{D}_{\tau(s)}}\rt)ds\\
&+C_0\int^t_0\f{\langle s\rangle^{\f{5}{4}-\dl}}{\tau^{2}(s)}\lt(\|g(s)\|_{\mathcal{X}_{\tau(s)}}+\langle s\rangle^{1/4}\|g(s)\|_{\mathcal{D}_{\tau(s)}}\rt)\|g(s)\|_{\mathcal{Y}_{\tau(s)}}ds
\leq  \|g_0\|_{X_{\tau_0}}\leq \e_0,
\eali
\ees
and the tangentially analytical radius $\tau(t)\geq \f{1}{2}\tau_0$.
\end{proposition}

Before proving Proposition \ref{proapriori}, we give two lemmas which concern on bounds of $u^r,\ u^z$, $\phi$ in terms of $g$.

\subsection{Bounding of $u^r,\ u^z$, $\phi$ in terms of $g$}

\begin{lemma}\label{gcontrol1}
Let $(u^r,u^z)$ be the solution of \eqref{3dasprandtl3}, $\phi$ and $g$ be the functions defined in \eqref{defnphi} and \eqref{defng}. For any $n\in\bN,\ |\al|=n$ and $0\leq \la<1$, we have

\be\label{ephi0}
\bali
&\lt|\th_{\la}\angr^n\p^\al_h \phi\rt|\ls_{\la} \th_{\la-1}{\angt}^{\f{1}{4}}\| \th\angr^n\p^\al_h g\|_{L^{2}_z},
\eali
\ee


\be\label{ethetaur3}
\bali
&\lt|\th_{\la}\angr^n\p^\al_h u^r\rt|\ls_{\la} \lt|\th_{\la}\angr^n\p^\al_h g\rt|+\f{z}{\angt^{3/4}} \th_{\la-1}\| \th\angr^n\p^\al_h g\|_{L^{2}_z},
\eali
\ee
and

\be\label{ethetaur5}
\bali
\lt|\th_{\la}\angr^n\p^\al_h \p_z u^r\rt|\ls_{\la}& \f{z}{\angt}\lt|\th_{\la}\angr^n\p^\al_h g(z)\rt|+ \lt| \th_{\la}\angr^n\p^\al_h \p_z g\rt|\\
&+\lt(\f{1}{\angt}+\f{z^2}{\angt^2}\rt) \th_{\la-1}{\angt}^{\f{1}{4}}\| \th\angr^n\p^\al_h g\|_{L^{2}_z}.
\eali
\ee

\end{lemma}

\pf
 We only show the proof of that $n=0$ since the case $n>0$ follows the same line.   From $\eqref{3dasprandtl3}_2$, we have
 \bes
 r\p_r \int^\i_0 u_rdz +2\int^\i_0 u_rdz=-\int^\i_0 \p_z u_zdz=u_z(t,r,0)=0,
 \ees
which indicates that
\bes
r \int^\i_0 u^rdz=0.
\ees
Since when $r>0$, the above equality implies that $\int^\i_0 u^r dz=0$ for $r>0$, then continuity of $u^r$ indicates that
\bes
 \int^\i_0 u^rdz\equiv 0.
\ees
By the definition of $\phi$ and $g$ in \eqref{defnphi} and \eqref{defng}, we have
\be\label{phig}
\lt\{
\bali
&\p_z\phi+\f{ z}{2\langle t\rangle} \phi=g,\\
&\phi\big|_{z=0}=0.
\eali
\rt.
\ee
Solving the ODE, we get
\be\label{phif}
\phi(t,r,z)=\exp\lt(-\f{z^2}{4\angt}\rt)\int^z_0 g(t,r,\bar{z})\exp\lt(\f{\bar{z}^2}{4\angt}\rt) d\bar{z}.
\ee
For any $0\leq\la<1$, by multiplying the above equality with $\th_{\la}$, we have
\be\label{ethetaphi}
\bali
\th_{\la}\phi= \th_{\la-1}(z) \int^z_0 \th(\bar{z})g(\bar{z})\exp\lt(\f{1}{8\angt}(\bar{z}^2-z^2)\rt) d\bar{z}.
\eali
\ee
Differentiating \eqref{phif} on $z$ gives that
\be\label{urformula}
u^r(t,r,z)=\p_z\phi=-\f{z}{2\angt}\exp\lt(-\f{z^2}{4\angt}\rt)\int^z_0 g(t,r,\bar{z})\exp\lt(\f{\bar{z}^2}{4\angt}\rt) d\bar{z}+g.
\ee
Multiplying \eqref{urformula} by $\th_\la$ gives that
\be\label{ethetaur2}
\bali
\th_{\la}u^r=&\th_{\la}g-\f{z}{2\angt} \th_{\la-1}(z) \int^z_0 \th(\bar{z})g(\bar{z})\exp\lt(\f{1}{8\angt}(\bar{z}^2-z^2)\rt) d\bar{z}.
\eali
\ee
Differentiating \eqref{urformula} on $z$ and multiplying the resulted equation by $\th_\la$ give that
\be\label{zurest}
\bali
\th_{\la}\p_zu^r=&\th_{\la}\p_z g-\f{z}{2\angt} \th_{\la}g\\
 &-\lt(\f{1}{2\angt}-\f{ z^2}{4\angt^2} \rt)\th_{\la-1} \int^z_0 \th(\bar{z})g(\bar{z})\exp\lt(\f{1}{8\angt}(\bar{z}^2-z^2)\rt) d\bar{z}.
\eali
\ee
Using the fact that for any $\beta\geq 0$,
\bes
\sup_{\zeta\geqq 0} \zeta^\beta e^{-\zeta^2}\leq C_\beta,
\ees
we have
\bes
\lt| \lt(\f{z}{\s{\angt}}\rt)^\beta \th_{\la-1}\rt| \leq C_{\la,\beta}.
\ees
Moreover, by considering $0\leq\zeta\leq 1$ and $\zeta>1$, it is not hard to check that
\bes
e^{-\zeta^2}\int^\zeta_0 e^{\bar{\zeta}^2}d\bar{\zeta}\leq \f{2}{1+\zeta}.
\ees
Then a change of variable indicates that
\be\label{useineq}
 \int^z_0 \exp\lt(\f{1}{4\angt}(\bar{z}^2-z^2)\rt) d\bar{z}\leq \f{C}{1+\zeta} \s{\angt}.
\ee
In \eqref{ethetaphi}, by using H\"{o}lder inequality on $z$, we have

\be\label{ethetaphi1}
\bali
\lt|\th_{\la}\phi\rt|\leqq& \th_{\la-1}(z)\|\th g\|_{L^{2}_z} \lt( \int^z_0 \exp\lt(\f{1}{4\angt}(\bar{z}^2-z^2)\rt) d\bar{z} \rt)^{1/2}\\
           \ls & \th_{\la-1}\|\th g\|_{L^{2}_z}\angt^{1/4} (1+\zeta)^{-1/4}\\
           \ls & \th_{\la-1}\|\th g\|_{L^{2}_z}\angt^{1/4},
\eali
\ee
which is \eqref{ephi0} for $n=0$.\\
In \eqref{ethetaur2}, by using H\"{o}lder inequality and \eqref{useineq}, we have

\bes
\bali
\lt|\th_{\la}u^r\rt|\leqq& \lt|\th_{\la}g\rt|+\f{z}{\angt} \th_{\la-1}\|\th g\|_{L^2_z} \lt( \int^z_0 \exp\lt(\f{1}{4\angt}(\bar{z}^2-z^2)\rt) d\bar{z} \rt)^{1/2}\\
               \ls& \lt|\th_{\la}g\rt|+\f{z}{\angt^{3/4}} \th_{\la-1}\|\th g\|_{L^2_z},
\eali
\ees
which is \eqref{ethetaur3} for $n=0$.

%
%
%
In \eqref{zurest}, by using H\"{o}lder inequality and \eqref{useineq}, we have
\bes
\bali
\lt|\th_{\la}\p_zu^r\rt|\ls_{\la}& \f{z}{\angt}\lt|\th_{\la}g\rt|+ \lt|\th_{\la}\p_z g\rt|\\
&+\lt(\f{1}{\angt}+\f{z^2}{\angt^2}\rt) \th_{\la-1}{\angt}^{\f{1}{4}}\|\th g\|_{L^{2}_z}(1+\zeta)^{-1/2}\\
\ls_\la & \f{z}{\angt}\lt|\th_{\la}g\rt|+ \lt|\th_{\la}\p_z g\rt|+\lt(\f{1}{\angt}+\f{z^2}{\angt^2}\rt) \th_{\la-1}{\angt}^{\f{1}{4}}\|\th g\|_{L^{2}_z},
\eali
\ees
which is \eqref{ethetaur5} for $n=0$.

%

By applying $\angr^n \p^\al_h$ to \eqref{ethetaphi}, \eqref{ethetaur2} and \eqref{zurest}, the above derivation from \eqref{ethetaphi1} also stand by replacing $\phi,\ u^r,\ \p_z u^r$ and $g$ by $\angr^n \p^\al_h \phi$, $\angr^n \p^\al_h u^r$, $\angr^n \p^\al_h \p_z u^r$ and  $\angr^n \p^\al_h g$, respectively.

Based on the rough estimates in Lemma \ref{gcontrol1}, we have the following much more subtle integration controls of $u^r$, $u^z$ and $\phi$ in terms of the weighted $L^2$ norm of $g$.

\begin{lemma}[Bounding of $u^r,\ u^z$, $\phi$ in terms of $g$]\label{gcontrol2}
For any $n\in\bN,\ |\al|=n$ and $0\leq \la<1$, we have the following estimates
%

\be\label{lphiesti1}
\bali
\lt\|\th_{\la}\angr^n \p^\al_h\phi \rt\|_{L^2_z}\ls_\la& \angt^{1/2}\lt\|\th\angr^n\p^\al_h g\rt\|_{L^2_z},
\eali
\ee
\be\label{luresti0}
\|\th_{\la}\angr^n \p^\al_h u^r\|_{L^2}\ls_\la \|\th\angr^n \p^\al_h g\|_{L^2},
\ee
\be\label{luresti-1}
\sum_{|\al|=n}\|\th_{\la}\angr^n \p^\al_h u^r\|_{L^\i_hL^2_z}\ls_\la (n+1)^2\sum^{n+2}_{|\al|=n}\|\th\angr^{|\al|} \p^\al_h g\|_{L^2},
\ee
\be\label{luresti1}
\bali
\lt\|\th_{\la}\angr^n \p^\al_h u^r\rt\|_{L^2_hL^\i_z}\ls_\la& \lt\|\th\angr^n \p^\al_h(g,\p_zg)\rt\|_{L^2},
\eali
\ee
\be\label{luresti2}
\bali
\sum_{|\al|=n}\lt\|\th_{\la}\angr^n \p^\al_h u^r\rt\|_{L^\i_hL^\i_z}\ls_\la (n+1)^2 \sum^{n+2}_{|\al|=n}\lt\|\th\angr^{|\al|} \p^{\al}_h(g,\p_z g)\rt\|_{L^2},
\eali
\ee
\be\label{luzesti1}
\bali
\lt\|\th_{\la}\angr^n \p^\al_h u^z\rt\|_{L^2_hL^\i_z}\ls_\la& \angt^{1/4} \lt\| \th\angr^n \p^\al_h (r\p_r g,g)\rt\|_{L^2},
\eali
\ee
\be\label{luzesti2}
\bali
\sum_{|\al|=n}\lt\|\th_{\la}\angr^n \p^\al_h u^z\rt\|_{{L^\i_hL^\i_z}}\ls_\la& (n+1)^2\angt^{1/4}\sum^{n+2}_{|\al|=n} \lt\|\th\angr^{|\al|} \p^{\al}_h (r\p_r g, g)\rt\|_{L^2},
\eali
\ee
\be\label{lpuzesti1}
\bali
\lt\|\th_{\la}\angr^n\p^\al_h \p_z u^r\rt\|_{L^2}\ls_{\la}  \angt^{-1/2}\lt\|\th\angr^n\p^\al_h g\rt\|_{L^2}+\lt\| \th \angr^n\p^\al_h \p_z g\rt\|_{L^2},
\eali
\ee
\be\label{lpuzesti2}
\bali
\sum_{|\al|=n}\lt\|\th_{\la}\angr^n \p^\al_h \p_zu^r \rt\|_{L^\i_hL^2_z}\ls_\la&(n+1)^2\sum^{n+2}_{|\al|=n} \lt(\angt^{-1/2}\lt\|\th\angr^{|\al|}\p^{\al}_h g\rt\|_{L^2}+\lt\| \th \angr^{|\al|}\p^{\al}_h \p_z g\rt\|_{L^2}\rt).
\eali
\ee

\end{lemma}

\begin{proof}

From \eqref{ephi0}, we have

\bes
\bali
\lt\|\th_{\la}\angr^n\p^\al_h \phi\rt\|_{L^2_z}\ls_{\la}& \|\th_{\la-1}\|_{L^2_z}{\angt}^{\f{1}{4}}\| \th\angr^n\p^\al_h g\|_{L^{2}_z}\\
 \ls_\la& {\angt}^{\f{1}{2}}\| \th\angr^n\p^\al_h g\|_{L^{2}_z},
\eali
\ees
where we have used the fact that when $\la-1<0$,
\bes
 \|\th_{\la-1}\|_{L^2_z}\ls_\la \angt^{1/4}.
\ees
Hence, we have obtained \eqref{lphiesti1}.

From \eqref{ethetaur3},  we have

\be\label{url2}
\bali
\lt\|\th_{\la}\angr^n\p^\al_h u^r\rt\|_{L^2}\ls_{\la}& \lt\|\th_{\la}\angr^n\p^\al_h g\rt\|_{L^2}+\|\f{z}{\angt^{3/4}} \th_{\la-1}\|_{L^2_z}\| \th\angr^n\p^\al_h g\|_{L^2}\\
 \ls_{\la}& \lt\|\th \angr^n\p^\al_h g\rt\|_{L^2},
\eali
\ee
which is \eqref{luresti0}.

Using the two dimensional Sobolev inequality

\bes
\|f\|_{L^\i_{h}}\ls \|f\|_{L^2_{h}} +\|\p^2_hf\|_{L^2_{h}},
\ees

we have
\be\label{eestiur}
\bali
\lt\|\th_{\la}\angr^n \p^\al_h u^r\rt\|_{L^\i_hL^2_z}\ls& \lt\|\th_{\la}\angr^n \p^\al_h u^r\rt\|_{L^2_hL^2_z}+\lt\|\th_{\la}\p^2_h\lt[\angr^n \p^\al_h u^r \rt]\rt\|_{L^2_hL^2_z}.
\eali
\ee

It is easy to show that for $n\in \bN/\{0\}$,
\be\label{sobolevx}
\lt|\p^2_h\lt[\angr^n \p^\al_h u^r \rt]\rt|\ls \f{(n+1)^2} {\angr^2} \sum^2_{|\g|=0}|\angr^{n+|\g|}\p^{\al+\g}_h u^r|.
\ee
Inserting \eqref{sobolevx} into \eqref{eestiur} and sum over $|\al|=n$, we have
\be\label{esobur}
\bali
\sum_{|\al|=n}\lt\|\th_{\la}\angr^n \p^\al_h u^r \rt\|_{L^\i_hL^2_z}\ls& (n+1)^2\sum^{n+2}_{|\al|=n}\lt\|\th_{\la}\angr^{|\al|} \p^{\al}_h u^r \rt\|_{L^2}.
\eali
\ee
Inserting \eqref{url2} into \eqref{esobur}, we obtain \eqref{luresti-1}.

Also from \eqref{ethetaur3},  we have
\be\label{euresti2}
\bali
\lt\|\th_{\la}\angr^n \p^\al_h u^r\rt\|_{L^\i_z}\ls &  \lt\|\th_{\la}\angr^n\p^\al_h g\rt\|_{L^\i_z}+\|\f{z}{\angt^{3/4}} \th_{\la-1}\|_{L^\i_z}\| \th\angr^n\p^\al_h g\|_{L^{2}}\\
                                \ls_{\la} &\lt\|\th_{\la}\angr^n\p^\al_h g\rt\|_{L^\i_z}+\angt^{-1/4} \|\th\angr^n\p^\al_h g\|_{L^{2}}.
\eali
\ee
Using one dimensional Sobolev embedding
\bes
\bali
\lt\|\th_{\la}\angr^n \p^\al_h g\rt\|_{L^\i_z}\ls& \lt\|\th_{\la}\angr^n \p^\al_h g\rt\|^{1/2}_{L^2_z}\lt\|\p_z\lt(\th_{\la}\angr^n \p^\al_h g \rt)\rt\|^{1/2}_{L^2_z}\\
 \ls&\lt\|\th_{\la}\angr^n \p^\al_hg\rt\|^{1/2}_{L^2_z}\lt[\lt\|\th_{\la}\angr^n \p^\al_h\p_z g\rt\|_{L^2_z}+\lt\|\f{z}{\angt}\th_{\la}\angr^n \p^\al_hg\rt\|_{L^2_z}\rt]^{1/2}\\
 \ls_{\la}& \lt\|\th\angr^n \p^\al_h g\rt\|_{L^2_z}+\lt\|\th\angr^n \p^\al_h \p_z g\rt\|_{L^2_z}.
\eali
\ees
Inserting the above inequality into \eqref{euresti2}, we can have
\be\label{euresti3}
\bali
\lt\|\th_{\la}\angr^n \p^\al_h u^r\rt\|_{L^\i_z}\ls_{\la} & \lt\|\th\angr^n \p^\al_h (g,\p_z g) \rt\|_{L^2_z}.
\eali
\ee

The bound \eqref{luresti1} follows by taking $L^2$ norms in $x,y$ variables of the above inequality \eqref{euresti3}.

The same as \eqref{esobur}, we can have
\be\label{esoburx}
\sum_{|\al|=n}\lt\|\th_{\la}\angr^n \p^\al_h u^r \rt\|_{L^\i_hL^\i_z}\ls (n+1)^2\sum^{n+2}_{|\al|=n}\lt\|\th_{\la}\angr^{|\al|} \p^{\al}_h u^r \rt\|_{L^2_hL^\i_z}.
\ee
Integrating \eqref{euresti3} on the tangential variables and inserting the resulted inequality into \eqref{esoburx}, we can get \eqref{luresti2}.
%
%

From the incompressible condition \eqref{3dasprandtl3}$_2$, we have
\bes
u^z(z)=-\int^\i_z \p_z u^z(\bar{z}) d\bar{z}=\int^\i_z (r\p_r u^r+2u^r)(\bar{z}) d\bar{z},
\ees
then we can get
\be\label{euzesti1}
\bali
\|\th_{\la}\angr^n \p^\al_h u^z\|_{L^\i_z}\leq& \|\th_{\la}\angr^n \p^\al_h(r\p_ru^r+2u^r)\|_{L^1_z}\\
\ls_{\la}& \|\th\angr^n \p^\al_h(r\p_ru^r+2u^r)\|_{L^2_z}\|\th_{\la-1}\|_{L^2_z}\\
\ls_{\la}& \angt^{1/4}\|\th\angr^n \p^\al_h(r\p_ru^r+2u^r)\|_{L^2_z}.
\eali
\ee
From \eqref{ethetaur3}, we have
\bes
\bali
&\|\th_{\la}\angr^n \p^\al_h(r\p_ru^r+2u^r)\|_{L^2_z}\\
\ls_{\la}  &\lt\|\th_{\la}\angr^n \p^\al_h(r\p_r+2)g\rt\|_{L^2_z}+\lt\|\f{z}{\angt^{3/4}}\th_{\la-1}\rt\|_{L^2_z}\|\th\angr^n \p^\al_h(r\p_r+2)g\|_{L^{2}_z}\\
\ls_{\la}  &\|\th\angr^n \p^\al_h (r\p_r+2)g\|_{L^{2}_z}.
\eali
\ees

The same as \eqref{esobur}, we can have
\be\label{esobuz}
\sum_{|\al|=n}\lt\|\th_{\la}\angr^n \p^\al_h u^z \rt\|_{L^\i_hL^\i_z}\ls_{\la} (n+1)^2\sum^{n+2}_{|\al|=n}\lt\|\th_{\la}\angr^{|\al|} \p^{\al}_h u^z \rt\|_{L^2_hL^\i_z}.
\ee
Inserting \eqref{luzesti1} into \eqref{esobuz}, we can get \eqref{luzesti2}.

Inserting the above inequality into \eqref{euzesti1} and then integrating the resulted equation in the tangential variables implies that
\bes
\bali
&\|\th_{\la}\angr^n\p^\al_h u^z\|_{L^2_hL^\i_z}\ls_{\la} \angt^{1/4}\|\th\angr^n\p^\al_h(r\p_r+2)g\|_{L^{2}},
\eali
\ees
which corresponds to \eqref{luzesti1}.

Then the same as \eqref{esobur} and using the estimate \eqref{luzesti1}, we can get \eqref{luzesti2}.

From \eqref{ethetaur5}, we can get
\bes
\bali
\lt\|\th_{\la}\angr^n\p^\al_h \p_z u^r\rt\|_{L^2}\ls_{\la}& \lt\| \f{z}{\angt}\th_{\la-1}\rt\|_{L^\i_z} \lt\|\th\angr^n\p^\al_h g\rt\|_{L^2}+ \lt\| \th \angr^n\p^\al_h \p_z g\rt\|_{L^2}\\
&+\lt\|\lt(\f{1}{\angt}+\f{z^2}{\angt^2}\rt) \th_{\la-1}\rt\|_{L^2_z}{\angt}^{\f{1}{4}} \| \th\angr^n\p^\al_h g\|_{L^{2}}\\
 \ls_\la&  \angt^{-1/2}\lt\|\th\angr^n\p^\al_h g\rt\|_{L^2}+\lt\| \th \angr^n\p^\al_h \p_z g\rt\|_{L^2}.
\eali
\ees
which is \eqref{lpuzesti1}.

Then almost the same as \eqref{esobur}, we can get
\bes
\bali
\sum_{|\al|=n}\lt\|\th_{\la}\angr^n \p^n_h \p_z u^r \rt\|_{L^\i_hL^2_z}\ls_{\la}& (n+1)^2\sum^{n+2}_{|\al|=n}\lt\|\th_{\la}\angr^{|\al|} \p^{\al}_h \p_zu^r \rt\|_{L^2_hL^2_z}\\
 \ls_{\la}&(n+1)^2\sum^{n+2}_{|\al|=n} \lt(\angt^{-1/2}\lt\|\th\angr^{|\al|}\p^{\al}_h g\rt\|_{L^2}+\lt\| \th \angr^{|\al|}\p^{\al}_h \p_z g\rt\|_{L^2}\rt).
\eali
\ees
which is \eqref{lpuzesti2}.
\end{proof}

\subsection{Weighted energy estimates for the good unknown $g$}

Now we perform the weighted energy estimates for the good unknown $g$.  Rewrite \eqref{3dasprandtl5}$_1$ as
\be\label{rg}
\lt\{
\bali
\partial_{t} g-\partial_{z}^{2} g+\f{1}{\angt}g=&-(ru^r\p_r +u^z\p_z)g-(u^r)^2+\f{1}{2\angt}u^z\p_z(z\phi)-\f{z}{\angt}u^r\phi\\
 &-\f{z}{2\angt}\int^{\i}_z (u^r)^2d\bar{z}+\f{z}{\angt}\int^{\i}_z \p_zu^r u^zd\bar{z}.
\eali
\rt.
\ee
Let $n\geq 0$ and $|\al|=n$. Applying $\angr^n\p^\al_h$ to $\eqref{rg}$ and multiplying the resulted equation with $\th^2\angr^n\p^\al_h g$, and then integrating over $\bR^3_+$ to give

\bes
\bali
&\f{1}{2}\f{\rm{d}}{\rm{d}t} \|\th\angr^n\p^\al_h g\|^2_{L^2}+ \|\th\angr^n\p^\al_h \p_z g\|^2_{L^2}+\f{3}{4\angt}\|\th \angr^n\p^\al_h g\|^2_{L^2}\\
=&- \int \th \angr^{n}\p^{\al}_h (u^r r\p_rg) \th\angr^n\p^\al_h  g- \int \th\angr^n\p^\al_h (u^z\p_z g) \th\angr^n\p^\al_h g\\
 &- \int \th \angr^{n}\p^{\al}_h (u^r)^2 \th\angr^n\p^\al_h g+\f{1}{2\angt} \int \th \angr^n\p^\al_h \lt(u^z\p_z(z\phi)\rt) \th\angr^n\p^\al_h g\\
 &-\f{1}{\angt}\int z\th  \angr^n\p^\al_h  (u^r\phi) \th\angr^n\p^\al_h g-\f{1}{2\angt} \int z\th\int^\i_z\angr^n\p^\al_h (u^r)^2d\bar{z} \th\angr^n\p^\al_h g\\
 &+\f{1}{\angt} \int z\th \int^\i_z \angr^n\p^\al_h (\p_z u^r u^z)d\bar{z} \th \angr^n\p^\al_h g\\
:=&\sum^7_{j=1}I^\al_{j}.
\eali
\ees
Here for a function $f(t,x,y,z)$, we have denoted $\int_{\bR^3_+} f(t,x,y,z)dxdydz$ simply by $\int f$ if no confusion is caused.

Dividing the above equality by $\|\th\angr^n\p^\al_h g\|_{L^2}$ and multiplying the resulted equation by $\tau^n(t) M_n$, then by summing for $|\al|=n$, we can get, for $n\geq 0$,
\be\label{weight2}
\bali
&\f{\rm{d}}{\rm{d}t} X_n+ \sum\limits_{|\al|=n}\f{ \|\th\angr^n\p^\al_h \p_z g\|^2_{L^2}}{\|\th \angr^n\p^\al_h g\|_{L^2}}+\f{3}{4\angt}X_n=\dot{\tau}(t)Y_n+\sum_{|\al|=n}\f{\tau^n(t) M_n}{\|\th \angr^n\p^\al_h g\|_{L^2}}\sum^7_{j=1}I^\al_{j},
\eali
\ee
where when $n=0$, we set $Y_0=0$.

Here we present a lemma to characterize the quantitative relation between $\|\th\angr^n\p^\al_h g\|^2_{L^2}$ and $\|\th \angr^n\p^\al_h \p_zg\|^2_{L^2}$.
\begin{lemma}\label{lpoincare}
Let $g$ be a smooth enough function in $x,y$ variables and belong to $H^1$ in $z$ variable, which decays to zero sufficiently fast as $z\rightarrow +\i$. Then we have
\be\label{weightpoincare}
\f{1}{2\angt}\|\th \angr^n\p^\al_h g\|^2_{L^2}\leq \|\th \angr^n\p^\al_h \p_zg\|^2_{L^2}.
\ee
\end{lemma}

The inequality \eqref{weightpoincare} is a special case of Treves inequality that can be found in \cite{Hormander:1985SPRINGER}. Proof of Lemma \ref{lpoincare} can be found in \cite[Lemma 3.1]{PZ:2021ARMA}. See also \cite[Lemma 3.3]{IV:2016ARMA}. Here, we omit the details.

Using \eqref{weightpoincare}, we can obtain from \eqref{weight2}
\be\label{weight3}
\f{\rm{d}}{\rm{d}t} X_n+ \f{1}{\s{2\angt}}D_n+\f{3}{4\angt}X_n\leq \dot{\tau}(t)Y_n+\sum_{|\al|=n}\f{\tau^n(t) M_n}{\|\th \angr^n\p^\al_h g\|_{L^2}}\sum^7_{j=1}I^\al_{j}.
\ee

\subsection{Proof of Proposition \ref{proapriori} and the main theorem}

First, we state a proposition concerning on the estimates of the nonlinear terms in \eqref{weight3}.

\begin{proposition}[Estimates of the nonlinear terms]\label{nonest}
For the nonlinear terms in \eqref{weight3}, we have the following estimate
\bes
\bali
&\sum_{n\geqq 0}\sum_{|\al|=n}\f{\tau^n(t) M_n}{\|\th\angr^n\p^\al_h g\|_{L^2}}\sum^7_{j=1}I^\al_{j}\\
 \leq& C\tau^{-2}(t)\lt(\|g\|_{\mathcal{X}_{\tau}}+\angt^{1/4}\|g\|_{\mathcal{D}_{\tau}}\rt)
 \|g\|_{\mathcal{Y}_{\tau}}
 +C\tau^{-2}(t)\lt(\|g\|_{\mathcal{X}_{\tau}}+\angt^{1/4}\|g\|_{\mathcal{D}_{\tau}}\rt)
 \|g\|_{\mathcal{X}_{\tau}}.
\eali
\ees
\end{proposition}

We postpone the proof of Proposition \ref{nonest} in Section \ref{pnonest} and continue to prove the a priori estimate in Proposition \ref{proapriori}.

\subsubsection*{Proof of Proposition \ref{proapriori}}

From \eqref{weight3}, {by summing on $n\geq 0$}, we get for a uniform constant $C_0$,
\be\label{ecomp1}
\bali
&\f{d}{dt}\|g\|_{\mathcal{X}_{\tau}}+\f{1}{\s{2\angt}}\|g\|_{\mathcal{D}_{\tau}}
+\f{3}{4{\angt}}\|g\|_{\mathcal{X}_{\tau}}\\
\leq& \lt(\dot{\tau}+C_0\tau^{-2}(t)\lt(\|g\|_{\mathcal{X}_{\tau}}+\angt^{1/4}
\|g\|_{\mathcal{D}_{\tau}}\rt)\rt)\|g\|_{\mathcal{Y}_{\tau}}
+C_0\tau^{-2}(t)\lt(\|g\|_{\mathcal{X}_{\tau}}+\angt^{1/4}\|g\|_{\mathcal{D}_{\tau}}\rt)
\|g\|_{\mathcal{X}_{\tau}}.
\eali
\ee
By using \eqref{weightpoincare}, for any small $\dl_1>0$, we have
\bes
\bali
\f{1}{\s{2\angt}}\|g\|_{\mathcal{D}_{\tau}}=& \f{\dl_1}{\s{2\angt}}\|g\|_{\mathcal{D}_{\tau}}+\f{(1-\dl_1)}{\s{2\angt}}\|g\|_{\mathcal{D}_{\tau}}\\
    \geq&\f{\dl_1}{\s{2\angt}}\|g\|_{\mathcal{D}_{\tau}}+\f{(1-\dl_1)}{{2\angt}}
    \|g\|_{\mathcal{X}_{\tau}}\\
    \geq& \f{\dl_1}{\s{2\angt}}\|g\|_{\mathcal{D}_{\tau}}+\f{\dl_1}{{\angt}}
    \|g\|_{\mathcal{X}_{\tau}}+\f{1-3\dl_1}{{2\angt}}\|g\|_{\mathcal{X}_{\tau}}.
\eali
\ees
Inserting the above inequality into \eqref{ecomp1}, we obtain that

\bes
\bali
&\f{d}{dt}\|g\|_{\mathcal{X}_{\tau}}+\f{\f{5}{4}-\f{3}{2}\dl_1}{{\angt}}
\|g\|_{\mathcal{X}_{\tau}}+\lt(\f{\dl_1}{{\angt}}\|g\|_{X_{\tau}}
+\f{\dl_1}{\s{2\angt}}\|g\|_{\mathcal{D}_{\tau}}\rt)\\
 \leq& \lt(\dot{\tau}+C_0\tau^{-2}(t)\lt(\|g\|_{\mathcal{X}_{\tau}}+\angt^{1/4}
 \|g\|_{\mathcal{D}_{\tau}}\rt)\rt)\|g\|_{\mathcal{Y}_{\tau}}
 +C_0\tau^{-2}(t)\lt(\|g\|_{\mathcal{X}_{\tau}}+\angt^{1/4}\|g\|_{\mathcal{D}_{\tau}}\rt)
 \|g\|_{\mathcal{X}_{\tau}}.
\eali
\ees
For $\dl\in(0,1/4]$, by choosing $\dl_1=\dl/3$, we have
\be\label{apriorx}
\bali
&\f{d}{dt}\|g\|_{\mathcal{X}_{\tau}}+\f{\f{5}{4}-\f{1}{2}\dl}{{\angt}}
\|g\|_{\mathcal{X}_{\tau}}+\f{\dl}{6}\lt(\f{1}{{\angt}}\|g\|_{\mathcal{X}_{\tau}}
+\f{1}{\s{\angt}}\|g\|_{\mathcal{D}_{\tau}}\rt)\\
\leq& \lt(\dot{\tau}+C_0\tau^{-2}(t)\lt(\|g\|_{\mathcal{X}_{\tau}}+\angt^{1/4}
\|g\|_{\mathcal{D}_{\tau}}\rt)\rt)\|g\|_{\mathcal{Y}_{\tau}}
 +C_0\tau^{-2}(t)\lt(\|g\|_{\mathcal{X}_{\tau}}+\angt^{1/4}\|g\|_{\mathcal{D}_{\tau}}\rt)
 \|g\|_{\mathcal{X}_{\tau}}.
\eali
\ee
Now, we assume the a prior assumption that for any $t>0$,
\be\label{aprior}
\langle t\rangle^{\f{5}{4}-\dl}\|g\|_{\mathcal{X}_{\tau}}\leq 2\e_0,\q \tau(t) \geq \f{1}{4}\tau_0.
\ee

Using this a priori assumption \eqref{aprior} and by choosing suitable $\tau(t)$ and  sufficiently small $\e_0$, depending on $\tau_0$ and $\dl$, we will show that
\be\label{aprior1}
\langle t\rangle^{\f{5}{4}-\dl}\|g\|_{\mathcal{X}_{\tau}}\leq \e_0,\q \tau(t) \geq \f{1}{2}\tau_0.
\ee
Then continuity argument insure that \eqref{aprior1} stands for any $t>0$.

First, inserting \eqref{aprior} into \eqref{apriorx}, we have
\bes
\bali
&\f{d}{dt}\|g\|_{\mathcal{X}_{\tau}}+\f{\f{5}{4}-\f{1}{2}\dl}{{\angt}}
\|g\|_{\mathcal{X}_{\tau}}+\f{\dl}{6}\lt(\f{1}{{\angt}}\|g\|_{\mathcal{X}_{\tau}}
+\f{1}{\s{\angt}}\|g\|_{\mathcal{D}_{\tau}}\rt)\\
\leq& \lt(\dot{\tau}+C_0\tau^{-2}(t)\lt(\|g\|_{\mathcal{X}_{\tau}}+\angt^{1/4}
\|g\|_{\mathcal{D}_{\tau}}\rt)\rt)\|g\|_{\mathcal{Y}_{\tau}}
+\f{32\e_0C_0}{\tau^2_0\angt^{5/4-\dl}}\lt(\|g\|_{\mathcal{X}_{\tau}}+\angt^{1/4}
\|g\|_{\mathcal{D}_{\tau}}\rt).
\eali
\ees
By choosing $\e_0$ such that $\f{32\e_0 C_0}{\tau^2_0}<\f{\dl}{12}$, then we can have
\be\label{apriorx2}
\bali
&\f{d}{dt}\|g\|_{\mathcal{X}_{\tau}}+\f{\f{5}{4}-\dl}{{\angt}}
\|g\|_{\mathcal{X}_{\tau}}+\f{\dl}{12}\lt(\f{1}{{\angt}}\|g\|_{\mathcal{X}_{\tau}}
+\f{1}{\s{\angt}}\|g\|_{\mathcal{D}_{\tau}}\rt)\\
\leq& \lt(\dot{\tau}+C_0\tau^{-2}(t)\lt(\|g\|_{\mathcal{X}_{\tau}}+\angt^{1/4}
\|g\|_{\mathcal{D}_{\tau}}\rt)\rt)\|g\|_{\mathcal{Y}_{\tau}}.
\eali
\ee
We choose $\tau(t)$ such that
\be\label{analr}
\dot{\tau}+\f{2C_0}{\tau^{2}(t)}\lt(\|g\|_{\mathcal{X}_{\tau}}+\angt^{1/4}
\|g\|_{\mathcal{D}_{\tau}}\rt)=0.
\ee
Then \eqref{apriorx2} indicates that
\be\label{apriorx5}
\bali
&\f{d}{dt}\lt(\angt^{\f{5}{4}-\dl}\|g\|_{\mathcal{X}_{\tau}}\rt)+\f{\dl}{12}
\lt(\angt^{\f{1}{4}-\dl}\|g\|_{\mathcal{X}_{\tau}}+\angt^{\f{3}{4}-\dl}
\|g\|_{\mathcal{D}_{\tau}}\rt)\\
&+\f{C_0\angt^{\f{5}{4}-\dl}}{\tau^{2}(t)}\lt(\|g\|_{\mathcal{X}_{\tau}}+\angt^{1/4}
\|g\|_{\mathcal{D}_{\tau}}\rt)\|g\|_{\mathcal{Y}_{\tau}}\leq0.
\eali
\ee
Integrating \eqref{apriorx5}, we can have
\be\label{apriorx3}
\bali
&\langle t\rangle^{\f{5}{4}-\dl}\|g\|_{\mathcal{X}_{\tau}}+\f{\dl}{12}\int^t_0\lt(\langle s\rangle^{\f{1}{4}-\dl}\|g\|_{\mathcal{X}_{\tau}}+\langle s\rangle^{\f{3}{4}-\dl}\|g\|_{\mathcal{D}_{\tau}}\rt)ds\\
&+C_0\int^t_0\f{\langle s\rangle^{\f{5}{4}-\dl}}{\tau^{2}(s)}\lt(\|g\|_{\mathcal{X}_{\tau}}+\langle s\rangle^{1/4}\|g\|_{\mathcal{D}_{\tau}}\rt)\|g\|_{\mathcal{Y}_{\tau}}ds
\leq  \|g_0\|_{\mathcal{X}_{\tau_0}}\leq \e_0,
\eali
\ee
which implies that
$$
\int^t_0\lt(\langle s\rangle^{\f{1}{4}-\dl}\|g\|_{\mathcal{X}_{\tau}}+\langle s\rangle^{\f{3}{4}-\dl}\|g\|_{\mathcal{D}_{\tau}}\rt)ds\\
\leq \f{12}{\dl}\e_0.
$$
Then from \eqref{analr}, we see that
\bes
\bali
\tau^{3}(t)=&\tau^{3}_0-6C_0\int^t_0\lt(\|g\|_{\mathcal{X}_{\tau}}+\langle s\rangle^{1/4}\|g\|_{D_{\tau}}\rt)ds\\
 \geq& \tau^{3}_0-\f{72C_0\e_0}{\dl}
 \geq \lt(\f{1}{2}\tau_0\rt)^{3},
\eali
\ees
by choosing small $\e_0$. Then by choosing small $\e_0$, depending on $\tau_0$ and $\dl$, we obtain \eqref{aprior1} and \eqref{apriorx3}, which finishes the proof of Proposition \ref{proapriori}.
\subsubsection*{End Proof of Theorem \ref{thmain}}

Combining the local existence and uniqueness of the tangentially analytical solutions in Theorem \ref{thlocal} and continuity argument, we can obtain the validity of Theorem \ref{thmain}.

\section{Technical estimates of the nonlinear terms}\label{pnonest}

In this section, we give the technical estimates for the nonlinear terms on the righthand of \eqref{weight3}. When summing over $n\geq 0$, we can get the following tangentially analytical estimates for the nonlinear terms.

\begin{lemma}[Estimates of the nonlinear terms separately]\label{lnonlinear}
We have the following estimates for the the nonlinear terms on the righthand of \eqref{weight3}.

\be\label{esti1i5c}
\bali
\sum\limits_{n\geq 0}\sum\limits_{|\al|=n}\f{|I^\al_1|\tau^n(t) M_n}{\|\th \angr^n\p^\al_h g\|_{L^2}}\ls \tau^{-2}\lt(\|g\|_{\mathcal{X}_\tau}+\|g\|_{\mathcal{D}_\tau}\rt) \|g\|_{\mathcal{Y}_\tau},
\eali
\ee
\be\label{estii2c}
\bali
\sum\limits_{n\geq 0}\sum\limits_{|\al|=n}\f{|I^\al_2|\tau^n(t) M_n}{\|\th \angr^n\p^\al_h g\|_{L^2}}\ls \tau^{-2}\angt^{\f{1}{4}}\lt(\|g\|_{\mathcal{X}_\tau}+\|g\|_{\mathcal{Y}_\tau}\rt)
\|g\|_{\mathcal{D}_\tau},
\eali
\ee
\be\label{estii3i6c}
\bali
\sum\limits_{n\geq 0}\sum\limits_{|\al|=n}\f{|I^\al_3|\tau^n(t) M_n}{\|\th \angr^n\p^\al_h g\|_{L^2}}+\sum\limits_{n\geq 0}\sum\limits_{|\al|=n}\f{|I^\al_6|\tau^n(t) M_n}{\|\th \angr^n\p^\al_h g\|_{L^2}}\ls \tau^{-2}\lt(\|g\|_{\mathcal{X}_\tau}+\|g\|_{\mathcal{D}_\tau}\rt) \|g\|_{\mathcal{X}_\tau},
\eali
\ee
\be\label{estii4c}
\bali
\sum\limits_{n\geq 0}\sum\limits_{|\al|=n}\f{|I^\al_4|\tau^n(t) M_n}{\|\th \angr^n\p^\al_h g\|_{L^2}}\ls \tau^{-2}\angt^{-1/4}\lt(|g\|_{\mathcal{X}_\tau}+\|g\|_{\mathcal{Y}_\tau}\rt)
\|g\|_{\mathcal{X}_\tau},
\eali
\ee
\be\label{estii5c}
\bali
\sum\limits_{n\geq 0}\sum\limits_{|\al|=n}\f{|I^\al_5|\tau^n(t) M_n}{\|\th \angr^n\p^\al_h g\|_{L^2}}\ls \tau^{-2}\lt(\|g\|_{\mathcal{X}_\tau}+\|g\|_{\mathcal{D}_\tau}\rt) \|g\|_{\mathcal{X}_\tau},
\eali
\ee
\be\label{estii7c}
\bali
\sum\limits_{n\geq 0}\sum\limits_{|\al|=n}\f{|I^\al_7|\tau^n(t) M_n}{\|\th \angr^n\p^\al_h g\|_{L^2}}\ls \tau^{-2}\lt(\angt^{-1/4}\|g\|_{\mathcal{X}_\tau}+\angt^{1/4}\|g\|_{\mathcal{D}_\tau}\rt) \lt(\|g\|_{\mathcal{X}_\tau}+\|g\|_{\mathcal{Y}_\tau}\rt).
\eali
\ee
\end{lemma}

\begin{proof} Before the proof, we give the following simple claim.

\noindent{\bf Claim.} For any $k\in\bN$, $1\leq p,q\leq +\i$,
\be\label{commutator}
\sum\limits_{|\al|=k}\|\th \angr^{k}\p^\al_h (r\p_rg)\|_{L^p_hL^q_z}\ls \sum\limits_{|\al|=k+1}\|\th \angr^{k+1}\p^\al_h g\|_{L^p_hL^q_z}+k\sum\limits_{|\al|=k}\|\th \angr^{k}\p^\al_h g\|_{L^p_hL^q_z}.
\ee
\begin{proof}[Proof of the claim] Without loss of generality, we assume $k\geq 1$, since the claim is obviously stand for $k=0$. We write $r\p_r =x\p_x+y\p_y:=x_h \p_h$. Then using Leibniz formula, we have
\be\label{com1}
\bali
\lt|\angr^{k}\p^\al_h (r\p_rg)\rt|=&\lt|\angr^{k}\p^\al_h (x_h \p_h g)\rt|\\
         =& \lt|\angr^{k}x_h \p^\al_h\p_h g +\sum_{\beta\leq \al,|\beta|=1}\angr^{k} \lt(\al\atop \beta\rt)\p^{\al-\beta}_h \p_h g \p^\beta_h  x_h\rt|\\
         \leq& \angr^{k+1}| \p^\al_h\p_h g|+2k\angr^{k}|\p^{\al}_h g|.
\eali
\ee
Then from \eqref{com1}, we can easily obtain \eqref{commutator}.
\end{proof}

In later calculations, for multi-indices $\al,\beta$ with $\beta\leq \al$, we will frequently use
\be\label{indexineq0}
\lt(\al\atop \beta\rt)\leq \lt(|\al|\atop |\beta|\rt),\ \ \
\sum\limits_{|\al|=n}\sum\limits_{|\beta|=k,\beta\leq \al} a_\beta b_{\al-\beta}=\lt(\sum\limits_{|\beta|=k}a_\beta\rt)\lt(\sum\limits_{|\g|=n-k} b_\gamma\rt)
\ee
for all sequences $\{a_\beta\}$ and  $\{b_\gamma\}$.

Now we are ready to prove Lemma \ref{lnonlinear}.

\noindent{\bf Estimate for term $I_1$.}
For the term $I_1$, by using \eqref{indexineq0}, we have
\bes
\bali
\sum_{|\al|=n}\f{|I^n_1|\tau^n(t) M_n}{\|\th \angr^n\p^\al_h g\|_{L^2}}
\leq&\tau^n(t) M_n\sum^{[n/2]}_{k=0} \lt(n\atop k\rt) \lt(\sum_{|\g|=n-k}\|\angr^{n-k}\p^{\g}_hu^r\|_{L^2_hL^\i_z}\rt)\lt(\sum_{|\beta|=k}\|\th\angr^{k}\p^\beta_h (r\p_r g)\|_{L^\i_hL^2_z}\rt)\\
& +\tau^n(t) M_n\sum^{n}_{k=[n/2]+1} \lt(n\atop k\rt) \lt(\sum_{|\g|=n-k}\|\angr^{n-k}\p^{\g}_hu^r\|_{L^\i}\rt)\lt(\sum_{|\beta|=k}\|\th\angr^{k}\p^\beta_h (r\p_r g)\|_{L^2}\rt).
\eali
\ees
Then by using \eqref{luresti1} and \eqref{luresti2}, and noting that $M_n\lt(n\atop k\rt)=\f{(n+1)^4}{(n-k)!k!}$, we have
{\small
\be\label{termi1}
\bali
&\sum_{|\al|=n}\f{|I^n_1|\tau^n(t) M_n}{\|\th \angr^n\p^\al_h g\|_{L^2}}
\ls \sum^{[n/2]}_{k=0} \lt(X_{n-k}+D_{n-k}\rt) \f{\tau^k}{k!}\sum_{|\beta|=k}\|\th\angr^{k}\p^\beta_h (r\p_r g)\|_{L^\i_hL^2_z}\\
&\q +\tau^{-2}\sum^{n}_{k=[n/2]+1}\sum\limits^{2}_{i=0}\lt(X_{n-k+i}+D_{n-k+i}\rt)\f{\tau^k(k+1)^4}{k!}\sum_{|\beta|=k}\|\th\angr^{k}\p^\beta_h (r\p_r g)\|_{L^2}.
\eali
\ee

}

Then by the same Sobolev embedding estimate as in \eqref{esobur} and using \eqref{commutator}, we can get
\bes
\bali
\sum_{|\beta|=k}\|\th\angr^{k}\p^\beta_h (r\p_r g)\|_{L^\i_hL^2_z}\ls& (k+1)^2\sum^{k+2}_{|\beta|=k}\|\th\angr^{|\beta|}\p^\beta_h (r\p_r g)\|_{L^2}\\
 \ls& (k+1)^2\sum^{k+3}_{|\beta|=k+1}\|\th\angr^{|\beta|}\p^\beta_h g\|_{L^2}+(k+1)^2|\beta|\sum^{k+2}_{|\beta|=k}\|\th\angr^{|\beta|}\p^\beta_h g\|_{L^2}.
\eali
\ees
Then it is not hard to check that
\be\label{ergr1}
\f{\tau^k}{k!}\sum_{|\beta|=k}\|\th\angr^{k}\p^\beta_h (r\p_r g)\|_{L^\i_hL^2_z}\ls \tau^{-2}\sum^{3}_{i=0}{Y_{k+i}},
\ee
where, when $k=i=0,$ we have set $Y_0=0$.\\
Also by using \eqref{commutator}, we can obtain
\be\label{ergr2}
\f{\tau^k(k+1)^4}{k!}\sum_{|\beta|=k}\|\th\angr^{k}\p^\beta_h (r\p_r g)\|_{L^2}\ls Y_{k}+Y_{k+1},
\ee
where we used that $\tau\leq \tau_0$ since later we will chosen $\tau(t)$ to be a decreased function of $t$ .

Inserting \eqref{ergr1} and \eqref{ergr2} into \eqref{termi1}, we can get
\be\label{estitermi1}
\bali
\sum_{|\al|=n}\f{|I^\al_1|\tau^n(t) M_n}{\|\th \angr^n \p^\al_h g\|_{L^2}}
\leq &\tau^{-2}\sum^{n}_{k=0} \sum\limits^{2}_{i=0}\lt(X_{n-k+i}+D_{n-k+i}\rt) \sum^{3}_{i=0}Y_{k+i}.
\eali
\ee
Then by using the following inequality
\be\label{sum}
\sum_{n\geq 0} \sum^n_{k=0} a_{n-k}b_{k}\leq \sum_{k\geq 0} a_k \sum_{j\geq 0} b_j,
\ee
we can get from \eqref{estitermi1},
\bes
\bali
\sum_{n\geq 0}\sum_{|\al|=n}\f{|I^\al_1|\tau^n(t) M_n}{\|\th \angr^n\p^\al_h g\|_{L^2}}\ls \tau^{-2}\sum_{k\geq0}\lt(X_{k}+D_{k}\rt) \sum_{k\geq 0}Y_{k}=\tau^{-2}\lt(\|g\|_{\mathcal{X}_\tau}+\|g\|_{\mathcal{D}_\tau}\rt) \|g\|_{\mathcal{Y}_\tau},
\eali
\ees
which is \eqref{esti1i5c} for term $I_1$.

\noindent{\bf Estimate for term $I_2$.}
Now we come to estimate term $I_2$. By using \eqref{indexineq0}, we have
\be\label{termi2zero}
\bali
\sum_{|\al|=n}\f{|I^\al_2|\tau^n(t) M_n}{\|\th \angr^n\p^\al_h g\|_{L^2}}
\leq& \tau^n(t) M_n\sum^{[n/2]}_{k=0} \lt(n\atop k\rt) \sum_{|\g|=n-k}\|\angr^{n-k}\p^{\g}_hu^z\|_{L^2_hL^\i_z}\sum_{|\beta|=k}\|\th\angr^{k}\p^\beta_h \p_z g\|_{L^\i_hL^2_z}\\
&+\tau^n(t) M_n\sum^{n}_{k=[n/2]+1} \lt(n\atop k\rt)\sum_{|\g|=n-k}\|\angr^{n-k}\p^{\g}_hu^z\|_{L^\i}\sum_{|\beta|=k}\|\th\angr^{k}\p^\beta_h \p_z g\|_{L^2}.
\eali
\ee

Then by using \eqref{luzesti1} and \eqref{luzesti2}, and noting that $M_n\lt(n\atop k\rt)=\f{(n+1)^4}{(n-k)!k!}$, we have
\be\label{termi2}
\bali
&\q\sum_{|\al|=n}\f{|I^\al_2|\tau^n(t) M_n}{\|\th \angr^n\p^\al_h g\|_{L^2}}\\
&\ls \angt^{\f{1}{4}}\tau^n(t) \sum^{[n/2]}_{k=0} \f{(n-k+1)^4}{(n-k)!k!} \sum_{|\g|=n-k}\|\angr^{n-k}\p^{\g}_h (r\p_rg, g)\|_{L^2}\sum_{|\beta|=k}\|\th\angr^{k}\p^\beta_h \p_z g\|_{L^\i_hL^2_z}\\
&\q +\angt^{\f{1}{4}}\tau^n(t)\sum^{n}_{k=[n/2]+1} \f{(k+1)^4(n-k+1)^2}{(n-k)!k!}\sum^{n-k+2}_{|\g|= n-k}\|\angr^{|\g|}\p^{\g}_h(r\p_rg, g)\|_{L^2}\sum_{|\beta|=k}\|\th\angr^{k}\p^\beta_h \p_z g\|_{L^2}.
\eali
\ee
Noting that by using Sobolev embedding, we have
\be\label{termi2f}
\bali
\f{1}{k!}\|\th\angr^{k}\p^\beta_h \p_z g\|_{L^\i_hL^2_z}\ls& \f{(k+1)^2}{k!}\sum^{k+2}_{|\beta|=k}\|\th\angr^{|\beta|}\p^\beta_h \p_z g\|_{L^2}
      \ls  \tau^{-2}\sum^2_{i=0} D_{k+i}.
\eali
\ee
Inserting \eqref{termi2f} into \eqref{termi2}, we obtain

\be\label{estii2s}
\bali
&\q\sum_{|\al|=n}\f{|I^\al_2|\tau^n(t) M_n}{\|\th \angr^n\p^\al_h g\|_{L^2}}\\
&\ls \angt^{1/4}\tau^{-2}\sum^{[n/2]}_{k=0} \f{\tau^{n-k}(n-k+1)^4}{(n-k)!} \sum_{|\g|=n-k}\|\angr^{n-k}\p^{\g}_h (r\p_rg, g)\|_{L^2}\sum^2_{i=0}D_{k+i}\\
&\q +\angt^{1/4}\sum^{n}_{k=[n/2]+1} \f{\tau^{n-k}(n-k+1)^2}{(n-k)!}\sum^{n-k+2}_{|\g|= n-k}\|\angr^{|\g|}\p^{\g}_h(r\p_rg, g)\|_{L^2}D_{k}.
\eali
\ee
We have that
\begin{align}\label{termi2second}
&\f{\tau^{n-k}(n-k+1)^4}{(n-k)!} \sum_{|\g|=n-k}\|\angr^{n-k}\p^{\g}_h (r\p_rg, g)\|_{L^2}\ls X_{n-k}+Y_{n-k+1}+Y_{n-k},\\
\label{termi2third}
&\f{\tau^{n-k}(n-k+1)^2}{(n-k)!}\sum^{n-k+2}_{|\g|= n-k}\|\angr^{|\g|}\p^{\g}_h(r\p_rg, g)\|_{L^2}\leq X_{n-k}+\tau^{-2}\sum^3_{i=0}Y_{n-k+i}.
\end{align}
Inserting the above two inequalities into \eqref{estii2s}, we can obtain

\be\label{estii2third}
\bali
\sum_{|\al|=n}\f{|I^\al_2|\tau^n(t) M_n}{\|\th \angr^n\p^\al_h g\|_{L^2}}\ls&\angt^{1/4}\tau^{-2}\sum^{n}_{k=0} \lt(X_{n-k}+\sum^3_{i=0}Y_{n-k+i}\rt)\sum^2_{i=0}D_{k+i}.
\eali
\ee
Summing \eqref{estii2third} over $n\geq 0$ and using \eqref{sum}, we can obtain \eqref{estii2c}.

\noindent{\bf Estimate for term $I_3$.}
Now we come to estimate term $I_3$. By using \eqref{indexineq0}, we have
\be\label{termi3zero}
\bali
&\q\sum_{|\al|=n}\f{|I^\al_3|\tau^n(t) M_n}{\|\th \angr^n\p^\al_h g\|_{L^2}}\\
&\leq \tau^n(t) M_n\sum^{[n/2]}_{k=0} \lt(n\atop k\rt) \sum_{|\g|=n-k}\|\th_{1/2}\angr^{n-k}\p^{\g}_hu^r\|_{L^2_hL^\i_z}\sum_{|\beta|=k}\|\th_{1/2}\angr^{k}\p^\beta_h u^r\|_{L^\i_hL^2_z}\\
&\q +\tau^n(t) M_n\sum^{n}_{k=[n/2]+1} \lt(n\atop k\rt)\sum_{|\g|=n-k}\|\th_{1/2}\angr^{n-k}\p^{\g}_hu^r\|_{L^\i}\sum_{|\beta|=k}\|\th_{1/2}\angr^{k}\p^\beta_h u^r\|_{L^2}.
\eali
\ee
Then by using \eqref{luresti0} to \eqref{luresti2}, and noting that $M_n\lt(n\atop k\rt)=\f{(n+1)^4}{(n-k)!k!}$, we have
{\small
\be\label{termi3four}
\bali
\sum_{|\al|=n}\f{|I^\al_3|\tau^n(t) M_n}{\|\th \angr^n\p^\al_h g\|_{L^2}}
\ls& \tau^{-2}\sum^{[n/2]}_{k=0} \lt(X_{n-k}+D_{n-k}\rt)\sum^2_{i=0}X_{k+i} +\tau^{-2}\sum^{n}_{i=[n/2]+1}\sum^2_{i=0}\lt(X_{n-k+i}+D_{n-k+i}\rt) X_k\\
&\ls \tau^{-2}\sum^{n}_{k=0} \sum^2_{i=0}\lt(X_{n-k+i}+D_{n-k+i}\rt) \sum^2_{i=0}X_{k+i}.
\eali
\ee
}
Summing \eqref{termi3four} over $n\geq 0$ and using \eqref{sum}, we can obtain \eqref{estii3i6c} for term $I_3$.

\noindent{\bf Estimate for term $I_4$.}
For the terms $I^n_4$, from \eqref{phig}$_1$, we first have
\bes
\p_z(z\phi) =(1-\f{z^2}{2\angt})\phi +z g.
\ees
Then from \eqref{lphiesti1}, we have, for $|\al|=k$,
\be\label{phil2g}
\bali
\|\th_{\la}\angr^k\p^\al_h \p_z(z\phi) \|_{L^2_z}\leq & \|\th_{\la}\angr^k\p^\al_h \phi \|_{L^2_z}+\|\th_{\la}\f{z^2}{\angt}\angr^k\p^\al_h \phi \|_{L^2_z}+\|\th_{\la}z\angr^k\p^\al_h g \|_{L^2_z}\\
\leq &\s{\angt}\|\th_{\la}\angr^k\p^\al_h g \|_{L^2_z}+\|\th_{\f{1+\la}{2}\al}\angr^k\p^\al_h\phi \|_{L^2_z}+\s{\angt}\|\th_{\f{1+\la}{2}}\angr^k\p^\al_h g \|_{L^2_z}\\
\leq &\s{\angt}\|\th \angr^k\p^\al_h g \|_{L^2_z}.
\eali
\ee
Now we come to estimate term $I_4$. By using \eqref{indexineq0} and \eqref{phil2g}, we have
\bes
\bali
&\q\sum_{|\al|=n}\f{|I^\al_4|\tau^n(t) M_n}{\|\th \angr^n\p^\al_h g\|_{L^2}}\\
&\leq \angt^{-1/2}\tau^n(t) M_n\sum_{|\al|=n}\sum_{\beta\leq \al\atop |\beta|\leq [n/2] }\lt(\al\atop \beta\rt) \|\th_{1/2}\angr^{n-|\beta|}\p^{\al-\beta}u^z\|_{L^2_hL^\i_z} \|\th\angr^{|\beta|}\p^\beta_h g\|_{L^\i_hL^2_z}\\
&\q +\angt^{-1/2}\tau^n(t) M_n \sum_{|\al|=n}\sum_{\beta\leq \al\atop |\beta|> [n/2] }\lt(\al\atop \beta\rt)\|\th_{1/2}\angr^{n-|\beta|}\p^{\al-\beta}u^z\|_{L^\i}\|\th\angr^{|\beta|}\p^\beta_h g\|_{L^2}.
\eali
\ees

Then almost the same estimate as in \eqref{termi2zero} by replacing $\p_z g$ with $g$ indicates a similar estimate as \eqref{estii2third} as follows.

\be\label{termi4}
\bali
\sum_{|\al|=n}\f{|I^\al_4|\tau^n(t) M_n}{\|\th \angr^n\p^\al_h g\|_{L^2}}\ls&\angt^{-1/4}\tau^{-2}\sum^{n}_{k=0} \lt(X_{n-k}+\sum^3_{i=0}Y_{n-k+i}\rt)\sum^2_{i=0}X_{k+i}.
\eali
\ee
%
Summing \eqref{termi4} over $n\geq 0$ and using \eqref{sum}, we can obtain \eqref{estii4c}.

\noindent{\bf Estimate for term $I_5$.}
It is easy to see that, from \eqref{lphiesti1},
\be\label{phil2g1}
 \|\th_{\la}\angr^k\p^\al_h (z\phi) \|_{L^2_z}\ls_\la \s{\angt} \|\th_{\f{1+\la}{2}}\angr^k\p^\al_h\phi\|_{L^2_z}\ls_\la \angt \|\th\angr^k\p^\al_h g\|_{L^2_z}.
\ee
By using \eqref{indexineq0} and \eqref{phil2g1}, we have
\bes
\bali
&\q\sum_{|\al|=n}\f{|I^\al_5|\tau^n(t) M_n}{\|\th \angr^n\p^\al_h g\|_{L^2}}\\
&\ls \tau^n(t) M_n\sum^{[n/2]}_{k=0} \lt(n\atop k\rt)\sum_{|\g|=n-k}\|\angr^{n-k}\p^{\g}_hu^r\|_{L^2_hL^\i_z}\sum_{|\beta|=k}\|\th\angr^{k}\p^\beta_h  g\|_{L^\i_hL^2_z}\\
&\q +\tau^n(t) M_n\sum^{n}_{k=[n/2]+1}\lt( n\atop k\rt) \sum_{|\g|=n-k}\|\angr^{n-k}\p^{\g}_hu^r\|_{L^\i}\sum_{|\beta|=k}\|\th\angr^{k}\p^\beta_h  g\|_{L^2}.
\eali
\ees
Then by using \eqref{luresti1} and \eqref{luresti2}, and noting that $M_n\lt(n\atop k\rt)=\f{(n+1)^4}{(n-k)!k!}$, we have
\be\label{termi5}
\bali
&\q\sum_{|\al|=n}\f{|I^\al_5|\tau^n(t) M_n}{\|\th \angr^n\p^\al_h g\|_{L^2}}\\
&\ls \sum^{[n/2]}_{k=0} \lt(X_{n-k}+D_{n-k}\rt) \f{\tau^k}{k!}\sum_{|\beta|=k}\|\th\angr^{k}\p^\beta_h g\|_{L^\i_hL^2_z} +\tau^{-2}\sum^{n}_{k=[n/2]+1}\sum\limits^{2}_{i=0}\lt(X_{n-k+i}+D_{n-k+i}\rt)X_k.
\eali
\ee
By using Sobolev embedding, it is easy to check that
\bes
\f{\tau^k}{k!}\sum_{|\beta|=k}\|\th\angr^{k}\p^\beta_h g\|_{L^\i_hL^2_z}\ls \tau^{-2}\sum^{2}_{i=0} X_{k+i}.
\ees
Inserting the above inequality into \eqref{termi5}, we can obtain
\be\label{termi5second}
\bali
&\sum_{|\al|=n}\f{|I^\al_5|\tau^n(t) M_n}{\|\th \angr^n\p^\al_h g\|_{L^2}}
&\ls \angt^{-1/2}\tau^{-2}\sum^{n}_{k=0}\sum\limits^{2}_{i=0}\lt(X_{n-k+i}+D_{n-k+i}\rt)\sum^{2}_{i=0} X_{k+i}.
\eali
\ee
Summing \eqref{termi5second} over $n\geq 0$ and using \eqref{sum}, we can obtain \eqref{estii5c}.

\noindent{\bf Estimate for term $I_6$.}
First, we have
\bes
\bali
\f{|I^\al_6|}{\|\th\angr^n\p^\al_h g\|_{L^2}}
\leq& \f{1}{\angt}\|z\th (z)\int^\i_z\angr^n\p^\al_h(u^r)^2(\bar{z})d\bar{z}\|_{L^2}\\
=& \f{1}{\angt}\|z\th_{-1/2}(z)\th_{3/2}(z)\int^\i_z\angr^n\p^\al_h(u^r)^2(\bar{z})d\bar{z}\|_{L^2}\\
\leq&\f{1}{\angt}\|z\th_{-1/2}(z)\|_{L^\i_hL^2_z}\|\th_{3/2}(z)\int^\i_z\angr^n\p^\al_h (u^r)^2(\bar{z})d\bar{z}\|_{L^2_hL^\i_z}\\
\ls&\angt^{-1/4}\|\th_{3/2}(z)\int^\i_z\angr^n\p^\al_h(u^r)^2(\bar{z})d\bar{z}\|_{L^2_hL^\i_z}.
\eali
\ees
While
\bes
\bali
&\left\|\th_{3/2}(z)\int^\i_z\angr^n\p^\al_h(u^r)^2(\bar{z})d\bar{z}\right\|_{L^\i_z}\\
\leq& \sup_{z\geq 0} \lt\{\th_{\f{3}{2}}(z)\lt(\int^\i_z \th_{-\f{7}{2}}(\bar{z})d\bar{z}\rt)^{1/2}\rt\}\|\th_{7/4}(z)\angr^n\p^\al_h(u^r)^2\|_{L^2_z}\\
\leq& \angt^{1/4}\|\th_{7/4}(z)\angr^n\p^\al_h(u^r)^2\|_{L^2_z}.
\eali
\ees
Then
\be\label{termi6}
\bali
\sum_{|\al|=n}\f{|I^\al_6|\tau(t)M_n}{\|\th (\angr\p_r)^n g\|_{L^2}}\leq&\tau(t)M_n\sum_{|\al|=n}\|\th_{7/4}(z)\angr^{n}\p^\al_h(u^r)^2\|_{L^2}.
\eali
\ee
The rest is the same as $I^\al_3$ in \eqref{termi3zero} by replacing $1/2$ with $7/8$ which indicates \eqref{estii3i6c} for term $I^\al_6$.

\noindent{\bf Estimate for term $I_7$.}
Repeating the proof for \eqref{termi6}, we can get

\bes
\bali
\sum_{|\al|=n}\f{|I^\al_7|\tau(t)M_n}{\|\th (\angr\p_r)^n g\|_{L^2}}\leq&\tau(t)M_n\sum_{|\al|=n}\|\th_{7/4}(z)\angr^{n}\p^\al_h(u^z\p_zu^r)\|_{L^2}.
\eali
\ees
By using \eqref{indexineq0}, we have
\bes
\bali
&\q\sum_{|\al|=n}\f{|I^\al_7|\tau^n(t) M_n}{\|\th \angr^n\p^\al_h g\|_{L^2}}\\
&\leq \tau^n(t) M_n\sum^{[n/2]}_{k=0} \lt(n\atop k\rt) \sum_{|\g|=n-k}\|\th_{7/8}\angr^{n-k}\p^{\g}_hu^z\|_{L^2_hL^\i_z}\sum_{|\beta|=k}\|\th_{7/8}\angr^{k}\p^\beta_h \p_z u^r\|_{L^\i_hL^2_z}\\
&\q +\tau^n(t) M_n\sum^{n}_{k=[n/2]+1} \lt(n\atop k\rt)\sum_{|\g|=n-k}\|\th_{7/8}\angr^{n-k}\p^{\g}_hu^z\|_{L^\i}\sum_{|\beta|=k}\|\th_{7/8}\angr^{k}\p^\beta_h \p_z u^r\|_{L^2}.
\eali
\ees
Then by using \eqref{luzesti1} to \eqref{lpuzesti2}, and noting that $M_n\lt(n\atop k\rt)=\f{(n+1)^4}{(n-k)!k!}$, we have
\bes
\bali
&\q\sum_{|\al|=n}\f{|I^\al_7|\tau^n(t) M_n}{\|\th \angr^n\p^\al_h g\|_{L^2}}\\
&\ls \angt^{\f{1}{4}}\tau^{-2} \sum^{[n/2]}_{k=0} \f{(n-k+1)^4\tau^{n-k}}{(n-k)!} \sum_{|\g|=n-k}\|\angr^{n-k}\p^{\g}_h (r\p_rg, g)\|_{L^2}\sum^{2}_{i=0} \lt(\angt^{-1/2} X_{k+i}+D_{k+i} \rt)\\
&\q +\angt^{\f{1}{4}}\sum^{n}_{k=[n/2]+1} \f{(n-k+1)^2\tau^{n-k}}{(n-k)!}\sum^{n-k+2}_{|\g|= n-k}\|\angr^{|\g|}\p^{\g}_h(r\p_rg, g)\|_{L^2}
\lt(\angt^{-1/2}X_{k}+D_{k}\rt).
\eali
\ees
Then using \eqref{termi2second} and \eqref{termi2third}, we obtain
\be\label{termi7four}
\bali
&\sum_{|\al|=n}\f{|I^\al_7|\tau^n(t) M_n}{\|\th \angr^n\p^\al_h g\|_{L^2}}\ls \angt^{\f{1}{4}}\tau^{-2} \sum^{n}_{k=0} \lt(X_{n-k}+ \sum^3_{i=0}Y_{n-k+i}\rt)\sum^{2}_{i=0} \lt(\angt^{-1/2} X_{k+i}+D_{k+i} \rt).\\
\eali
\ee
Summing \eqref{termi7four} over $n\geq 0$ and using \eqref{sum}, we can obtain \eqref{estii7c}.

\end{proof}

%
%
%

\section*{Acknowledgments}
\addcontentsline{toc}{section}{Acknowledgments}

\q  X. Pan is supported by National Natural Science Foundation of China (No. 12031006, No. 11801268) and C. J. Xu is supported by National Natural Science Foundation of China (No. 12031006) and the Fundamental Research Funds for the Central Universities of China.

\medskip

 {\footnotesize

%
%

 {\sc X. Pan: Department of Mathematics, Nanjing University of Aeronautics and Astronautics, Nanjing 210016, China}

  {\it E-mail address:}  xinghong\_87@nuaa.edu.cn

  \medskip

{\sc C. J. Xu: Department of Mathematics, Nanjing University of Aeronautics and Astronautics, Nanjing 210016, China\
and\
Universit\'e de Rouen, CNRS UMR 6085, Laboratoire de Math\'ematiques,
76801 Saint-Etienne du Rouvray, France}

  {\it E-mail address:}  xuchaojiang@nuaa.edu.cn

}
\end{document}